\documentclass[final,onefignum,onetabnum]{siamart171218}

\usepackage{amsfonts} 
\usepackage{graphicx}
\usepackage{epstopdf}
\usepackage{ulem}
\usepackage{enumitem}
\usepackage{amsmath,amssymb}

\def\UT{\widehat{A}^{\,T}}
\def\Re{\mathrm{Re}}
\def\Im{\mathrm{Im}}

\newcommand{\R}{\mathbb{R}}
\def\Pobl{P_{\,\mathcal{R}(BA),\,\mathcal{N}(BA)}}
\def\Poblb{P_{\,\mathcal{R}(B),\,\mathcal{N}(BA)}}

\newcommand{\BA}{\textsc{BA Iteration}}
\newcommand{\SBA}{\textsc{Shifted BA Iteration}}
\def\NR{[ \, N \ \, R \, ]}

\headers{Fixing Nonconvergence}{Dong, Hansen, Hochstenbach, and Riis}

\title{Fixing Nonconvergence of Algebraic Iterative Reconstruction with
  an Unmatched Backprojector\thanks{Submitted to the editors DATE.
  \funding{This work has been partially funded by Advanced grant 291405
  from the European Research Council.
  The third author has been supported by an NWO Vidi research grant.}}}

\author{Yiqiu Dong\thanks{Department of Applied Mathematics and Computer Science,
  Technical University of Denmark,
  DK-2800 Kgs. Lyngby, Denmark (\email{yido@dtu.dk},
  \url{http://www.compute.dtu.dk/\string~yido/},
  \email{pcha@dtu.dk}, \url{http://www.compute.dtu.dk/\string~pcha/},
  \email{nabr@dtu.dk}, \url{http://www.compute.dtu.dk/\string~nabr/}).}
\and Per~Christian~Hansen\footnotemark[2]
\and Michiel~E.~Hochstenbach\thanks{Department of Mathematics and Computer Science,
  TU Eindhoven, 5600 MB Eindhoven, The Netherlands
  (\url{http://www.win.tue.nl/\string~hochsten/})}
\and Nicolai Andr\'e Brogaard Riis\footnotemark[2]}

\begin{document}

\maketitle

\begin{abstract}
We consider algebraic iterative reconstruction methods with applications
in image reconstruction.
In particular, we are concerned with methods based on an unmatched
projector/backprojector pair; i.e., the backprojector is not the exact adjoint
or transpose of the forward projector.
Such situations are common in large-scale computed tomography, and we
consider the common situation where the
method does not converge due to the nonsymmetry of the iteration matrix.
We propose a modified algorithm that incorporates a small shift
parameter, and we give the conditions that guarantee convergence of this method
to a fixed point of a slightly perturbed problem.
We also give perturbation bounds for this fixed point.
Moreover, we discuss how to use Krylov subspace methods to
efficiently estimate the leftmost eigenvalue of a certain matrix
to select a proper shift parameter.
The modified algorithm is illustrated with test problems from
computed tomography.
\end{abstract}

\begin{keywords}
  Unmatched transpose, algebraic iterative reconstruction, perturbation theory,
  leftmost eigenvalue estimation, computed tomography
\end{keywords}

\begin{AMS}
65F10, 65F15, 65F22, 15A18, 15A60
\end{AMS}

\section{Introduction}

Algebraic iterative reconstruction techniques \cite{AIRtoolsII}\,---\,such
as the methods by Kaczmarz and Cimmino\,---\,play an important role in
solving inverse problems.
In particular, they are popular in computed tomography (CT) due to their great
flexibility with respect to the measurement geometry of the X-ray scanner and
their ability to handle very underdetermined problems.
This is in contrast to filtered backprojection and similar algorithms
\cite{Natterer} that rely on specific geometries and a large amount of data.
Algebraic iterative methods are also used successfully in other
image reconstruction problems such as image deblurring.

The fundamental mechanism behind these methods is known as semi-convergence
\cite{Natterer}.
During the initial iterations, the iterates approach the exact (and unattainable)
solution to the noise-free problem, while in later stages the iterates
converge to the undesired noisy solution.
The methods produce filtered, or regularized, solutions and the number of
iterations plays the role of a regularization parameter \cite{HansenDIP}.

The algebraic iterative methods, in their basic form as well as their
block versions, lend themselves very well to implementations that utilize
GPUs and other hardware accelerators, and where the coefficient matrix
is never stored; rather, the matrix-vector multiplications are computed on the fly.
This has led to an implementation paradigm that is routinely used in
software packages for computed tomography (such as ASTRA~\cite{ASTRA}
and TIGRE~\cite{TIGRE}), namely,
to focus on the computational speed of the matrix-vector multiplication.
However, this introduces a convergence issue that has largely been ignored.

To set the stage, we formulate the noise-free problem as
	\begin{equation}
	  \label{eq:Axb}
	  A\, \bar{x} = \bar{b} \ , \qquad A \in \R^{m\times n} \ ,
	\end{equation}
where the vectors $\bar{x}$ and $\bar{b}$ represent the exact image and the
noise-free data, while $A$
represents the forward model\,---\,known as the (forward) projection in CT
where both $m \ge n$ and $m < n$ are common.
The multiplication with $A^T$, the transpose of $A$, is known in CT as the
backprojection.
These two operations form the computational core of any algebraic iterative method
and therefore\,---\,to optimize for computational speed\,---\,the software developers
often choose different discretization schemes, and different model approximations,
for these operations \cite{vBS09}.
Consequently, in such software the backprojection corresponds to multiplication
with a matrix $\UT$ that is not the exact transpose of $A$.
We refer to this situation as having an \textsl{unmatched projector/backprojector pair},
and we call $\UT$ the \textsl{unmatched transpose}.

It appears that very little attention has been given to iterations based on
such unmatched pairs;
see \cite{ZG} for an early reference and \cite{ElHa18}, \cite{LRS18}
for two more recent ones.
The latter paper has introduced a methodology for analyzing such iterations and
formulated conditions for convergence, and it has been shown that an
unmatched projector/backprojector pair deteriorates the best possible solution
at the point of semi-convergence.

In the common situation of an unmatched projector/backprojector pair
where the convergence criterion from \cite{ElHa18} is not satisfied,
the iterations will fail to converge for noise-free data (although some kind of
semi-convergence may be observed experimentally for noisy data).
In this paper we show that a small and cost-efficient modification of the basic
algorithm can guarantee convergence to a solution of a slightly perturbed problem.
This ensures that the fast implementations of the forward projections
and backprojections can still be used without sacrificing convergence.
Moreover, to provide a theoretical foundation we
extend the convergence and perturbation analysis from \cite{ElHa18}
to the modified algorithm.

Our paper has been organized as follows.
Section 2 sets the stage by summarizing convergence results for a generic
iterative algorithm (proposed in \cite{ElHa18}) that allows an unmatched transpose.
Section 3 introduces the modified algorithm, gives the associated
convergence conditions, and discusses the perturbation theory for the
underlying problem.
The modified algorithm is based on the introduction of a shift parameter,
and in Section~\ref{sec:eigenvalues} we discuss how to efficiently estimate the leftmost
eigenvalue of a certain matrix, which defines this shift.
Finally, in Section~\ref{sec:numerical} we give numerical examples that illustrate the
new method for solving inverse problems, followed by some conclusions in
Section~\ref{sec:conclusions}.

We use the following notations:
$\| \cdot \|$ denotes the vector and matrix 2-norm,
$\mathcal{R}(A)$ and $\mathcal{N}(A)$ are the range and null space of $A$,
respectively, and we split a complex eigenvalue $\lambda_j$
into its real and imaginary parts denoted by $\Re(\lambda_j)$ and $\Im(\lambda_j)$,
respectively. For a vector $x$, we use $x^H$ for its conjugate transpose.

In Section \ref{sec:BAit} we
use notations and concepts associated with oblique projections
and oblique pseudoinverses
to obtain compact expression that would otherwise
be quite lengthy.  We refer to
\cite{Hans13} for details and geometric
interpretations of these quantities in relation to inverse problems.
Given two complementary subspaces $\mathcal{X}$ and $\mathcal{Y}$ of
$\mathbb{R}^m$ that intersect trivially,
and matrices $X$, $Y$, and $Y_0$ such that
  \[
    \mathcal{X} = \mathcal{R}(X) \ , \qquad
    \mathcal{Y} = \mathcal{R}(Y) \ , \qquad
    \mathcal{Y}^{\perp} = \mathcal{R}(Y_0) \ .
  \]
Then the $m \times m$ matrix $P_{\mathcal{X},\mathcal{Y}}$ denotes the
oblique projector onto $\mathcal{X}$ along $\mathcal{Y}$ which satisfies
  \[
    \forall x \in \mathcal{X} : P_{\mathcal{X},\mathcal{Y}} \, x = x \ , \qquad
    \forall y \in \mathcal{Y} : P_{\mathcal{X},\mathcal{Y}} \, y = 0 \ , \qquad
    \forall z \in \mathbb{R}^m : P_{\mathcal{X},\mathcal{Y}} \, z \in \mathcal{X} \ ,
  \]
and the projection matrix can be written as
  \begin{equation}
  \label{eq:oblproj}
    P_{\mathcal{X},\mathcal{Y}} = X \bigl( Y_0^T X \bigr)^{\!\dagger} \, Y_0^T \ ,
  \end{equation}
where $^\dagger$ denotes the Moore--Penrose pseudoinverse.
Moreover, if $X\in\mathbb{R}^{m \times n}$ then
the $n\times m$ matrix $X_{\mathcal{Y}}^{\dagger}$ denotes the oblique
pseudoinverse of $X$ along $\mathcal{Y}$, given by
  \begin{equation}
  \label{eq:oblpinv}
    X_{\mathcal{Y}}^{\dagger} = \left\{ \begin{array}{ll}
      X^{\dagger}\, P_{\mathcal{X},\mathcal{Y}}
         \ = \ \bigl( Y_0^T X \bigr)^{\!\dagger}\,  Y_0^T \ , & m \geq n \\[2mm]
      P_{\mathcal{Y},\mathcal{N}(X)} \, X^{\dagger}
       \ = \ Y\, (X\, Y)^{\dagger} \ , & m \ \leq \ n \ .
    \end{array} \right.
  \end{equation}

The case $m\geq n$ requires $\mathcal{R}(X)$ and $\mathcal{N}(Y)$ to be complementary,
while the case $m \leq n$ requires $\mathcal{N}(X)$ and $\mathcal{R}(Y)$ to be complementary.
If $\mathcal{Y} = \mathcal{X}^{\perp}$ then $P_{\mathcal{X},\mathcal{Y}}$
is the orthogonal projector on $\mathcal{X}$ and $X_{\mathcal{Y}}^{\dagger}$
is the ordinary pseudoinverse.

\section{The BA Iteration}
\label{sec:BAit}

When we consider noisy data $b = \bar{b} + e$, where the vector $e$ represents
the perturbation, then it is common to compute a (weighted) least squares solution.
In the simplest case with unit weights we can compute the solution
by means of the Landweber iteration (or gradient descent method)
with initial guess $x^0 = 0$:
  \begin{equation}
  \label{eq:Landweber}
    x^{k+1} = x^k + \omega \, A^T \, (b - A \, x^k) \ , \qquad
    k = 0,1,2,\ldots \ ,
  \end{equation}
where $\omega$ is a relaxation parameter satisfying
$0 < \omega < 2\,/\,\| A^T A \|$.

To analyze the behavior of this and similar algebraic iterative
methods with an unmatched transpose, we follow \cite{ElHa18} and
consider the \BA\ defined by
  \begin{equation}
  \label{eq:BAIteration}
    x^{k+1} = x^k + \omega \, B\, (b - A\, x^k) \ ,\qquad
    k = 0,1,2,\ldots \ ,
  \end{equation}
where different choices of the $n \times m$ matrix $B$ give unmatched-transpose versions of
known iterative methods.
For example, $B$ can be an unmatched transpose $\UT$ for
Landweber's method, or $B$ can be an unmatched approximation to
$A^T \mathrm{diag}(A\, A^T)^{-1}$ for Cimmino's method;
see \cite{AIRtoolsII} for an overview of methods.

The convergence of the \BA\ is governed by the (complex) eigenvalues
$\lambda_j$ of the matrix $BA$: \eqref{eq:BAIteration} converges if and only if
the relaxation parameter $\omega$ and all nonzero $\lambda_j$ satisfy
  \begin{equation}
  \label{eq:BAcc}
    0 < \omega < \frac{2\, \Re(\lambda_j)}{| \lambda_j |^2} \qquad
    \hbox{and} \qquad \Re(\lambda_j) > 0 \ ,
  \end{equation}
see \cite[Prop.~3.2]{ElHa18} for details.
Note that this specializes to the standard condition when $B=A^T$.

We will now investigate when the \BA\, \eqref{eq:BAIteration} has a unique fixed point.
From the definition of the \BA\ \eqref{eq:BAIteration}
with $x^0=0$ it follows that any fixed point $x^*$ must satisfy $B\, A\, x^* = B\, b$.
Moreover, it is also clear from \eqref{eq:BAIteration} that all iterates
$x^k \in \mathcal{R}(B)$; in particular this holds for $x^*$.
Therefore, we can write any fixed point in the form $x^* = B\, y$ for some vector $y \in \R^m$.
(This vector $y$ may not be unique, as one can add an arbitrary component
$z \in \mathcal{N}(B)$, but this is irrelevant for what is to follow.)
Inserting $x^* = B\, y$ we obtain an equation for $y$:
  \begin{equation}
  \label{eq:BABA}
    B\, A\, B\ y = B\, b \ .
  \end{equation}
Here, $B\,A$ is an operator from $\mathcal{R}(B)$ to itself.
Recall that $A \in \R^{m \times n}$ and $B \in \R^{n \times m}$, where
both $m \ge n$ and $m < n$ are possible in CT applications.

From \eqref{eq:BABA} it follows there is a unique fixed point $x \in \mathcal{R}(B)$
if and only if one of the following eight equivalent conditions holds.

\begin{proposition}
\label{prop:nonsing}
Consider the two matrices $A\in\mathbb{R}^{m\times n}$ and $B\in\mathbb{R}^{n\times m}$
with ranks $r_A, r_B \le \min \{m, n\}$ and with the singular value decompositions (SVDs)
$A = U_A \Sigma_A V_A^T$ and $B^T = U_B \Sigma_B V_B^T$.
The following statements are equivalent:
\begin{enumerate}[label=(\roman*)]
\item \label{first}
       $BA: \mathcal{R}(B) \to \mathcal{R}(B)$ is nonsingular
(meaning that $BA\, z = 0$ and $z \in \mathcal{R}(B)$ imply that $z=0$);
\item For every $b \in \R^m$, the equation $BAx = Bb$ has a
      unique solution $x \in \mathcal{R}(B)$;
\item $\mathcal{R}(B) \cap \mathcal{N}(BA) = \{0\}$;
\item \label{nullspaces}
      $\mathcal{N}(BAB) = \mathcal{N}(B)$;
\item \label{ranges}
      $\mathcal{R}(BAB) = \mathcal{R}(B)$;
\item \label{ranks}
      $\mathrm{rank}(BAB) = \mathrm{rank}(B)$;
\item \label{nonsingular}
      $A$ is nonsingular on $\mathcal{R}(B)$ and $B$ is nonsingular on $\mathcal{R}(AB)$;
\item \label{joint}
       $\mathcal{R}(B) \cap \mathcal{N}(A) = \{0\}$
      and $\mathcal{R}(AB) \cap \mathcal{N}(B) = \{0\}$;
\end{enumerate}
\end{proposition}

\begin{proof}
The equivalences follow relatively straightforwardly from the (dimensions of)
nullspaces and ranges of $A$, $B$, $BA$, and $AB$.
For \ref{nullspaces} we have $\mathcal{N}(BAB) \supseteq \mathcal{N}(B)$ with equality
if and only if \ref{first} holds.
For \ref{ranges} one has $\mathcal{R}(BAB) \subseteq \mathcal{R}(B)$ with equality
if and only if \ref{first} holds,
where the ranks in \ref{ranks} are the dimensions of the subspaces of \ref{ranges}.
A nonzero vector in $\mathcal{R}(B)$ cannot be mapped to zero by the
consecutive action of $BA$, which is stated in \ref{nonsingular} and \ref{joint}.
\end{proof}


We assume from now on that there is a unique solution,
and the next theorem provides specific expressions for this fixed point.
We will see that, even in the absence of noise, the fixed point of \eqref{eq:BAIteration}
is \textit{not} the exact solution $\bar{x}$.
One way to understand this\,---\,following the discussion in \cite{ElHa18}\,---\,is
by the fact that
the unmatched normal equations $BA\, x = B\, b$ may be viewed
as an oblique projection of $A\, x=b$, instead of the common orthogonal
projection underlying the normal equations $A^TA\, x = A^T\, b$.

\begin{theorem}
\label{prop:fixpoint}
Assume that $A$ and $B$ satisfy the criteria in Proposition~\ref{prop:nonsing}.
Then the fixed point $x^*$ of the \BA\ \eqref{eq:BAIteration}
with starting vector $x^0 = 0$ can be expressed in three ways:
  \begin{equation}
  \label{eq:fixpointnoisy}
    x^* = (B\, A)_{\mathcal{R}(B)}^{\dagger} B\, b =
    B\, (A\, B)_{\mathcal{N}(B)}^{\dagger}\, b =
    \Poblb\, A_{\mathcal{N}(B)}^{\dagger} b \ ,
  \end{equation}
and for noise-free data $\bar{b} = A\, \bar{x}$ the fixed point is given by
  \begin{equation}
  \label{eq:fixpoint}
    \bar{x}^* = \Pobl \, \bar{x} \ .
  \end{equation}
If $m \geq n$ and $A$ and $B$ have full rank then $x^* = (B\, A)^{-1} B\, b$
and $\bar{x}^* = \bar{x}$.
\end{theorem}

\begin{proof}
By writing the fixed point as $x^* = B\, y$ with $y \in \R^m$
it follows from \eqref{eq:BABA} that $x^* = B\, (B\, A\, B)^{\dagger} B\, b$,
and the first two expressions in \eqref{eq:fixpointnoisy} are obtained by recognizing
that $B\, ((B\, A)\, B)^{\dagger} = (B\, A)_{\mathcal{R}(B)}^{\dagger}$ and
$(B\, (A\, B))^{\dagger} B = (A\, B)_{\mathcal{N}(B)}^{\dagger}$;
cf.~(\ref{eq:oblpinv}).
We now introduce the SVD
  \[
    B\, A = U\, \Sigma\, V^T \ , \qquad
    \Sigma\in\R^{p\times p} \ , \quad U, V \in \R^{n\times p} \ ,
  \]
where $p = \mathrm{rank}(BA)$.
Inserting this in \eqref{eq:BABA} we get $U\, \Sigma\, V^T B\ y = B\, b$,
and by multiplying from the left with $\Sigma^{-1} U^T$ we obtain
$V^T B\ y = \Sigma^{-1} U^T B\, b$.
The solution $y$ of minimum norm is given by
  \[
    y = (V^T B)^{\dagger} \Sigma^{-1} U^T B\, b \ ,
  \]
and hence
  \[
    x^* = B\, (V^T B)^{\dagger} \Sigma^{-1} U^T B\, b
    = B\, (V^T B)^{\dagger} V^T V \Sigma^{-1} U^T B\, b
    = B\, (V^T B)^{\dagger} V^T (B\, A)^{\dagger} B\, b.
  \]
By recognizing $\Poblb = B\, (V^T B)^{\dagger} V^T$ as the oblique projector
onto $\mathcal{R}(B)$ along $\mathcal{N}(V^T)=\mathcal{N}(BA)$,
cf.~(\ref{eq:oblproj}), and
$A_{\mathcal{N}(B)}^{\dagger} = (B\, A)^{\dagger} B$ as the oblique pseudoinverse
of $A$ along $\mathcal{N}(B)$, cf.~(\ref{eq:oblpinv}), we obtain
the third expression in~\eqref{eq:fixpointnoisy}.

For the special case $b = \bar{b} = A\, \bar{x}$
the fixed point satisfies $\bar{x}^* \in \mathcal{R}(BA)$,
and hence we can write it as $\bar{x}^*=B\, A\, \bar{y}$ for some vector $\bar{y}\in \R^n$.
According to $B\, A\, \bar{x}^* = B\, \bar{b}=B\, A\, \bar{x}$, we obtain
  \[
    \bar{x}^* = B\, A \, (B\, A\, B\, A)^{\dagger} B\, A\, \bar{x} =
    B\, A \, \bigl( ((B\, A)^T)^T B\, A \bigr)^\dagger ((B\, A)^T)^T \bar{x},
  \]
and we recognize
$\Pobl = B\, A \bigl( ((B\, A)^T)^T B\, A \bigr)^\dagger ((B\, A)^T)^T$
as the oblique projector onto $\mathcal{R}(BA)$ along
$\mathcal{R}((BA)^T)^{\perp} = \mathcal{N}(BA)$.
\end{proof}

The sensitivity of the fixed point to perturbations of the
right-hand side can be characterized as follows.

\begin{corollary}
\label{cor:pb}
Let $\bar{x}^*$ and $x^*$ denote the fixed point of the \BA\ when applied
to the noise-free data $\bar{b}$ and the noisy data $b = \bar{b} + e$, respectively.
Then
  \[
    \| x^* - \bar{x}^* \| \leq \| \Poblb \| \,
    \| A_{\mathcal{N}(B)}^{\dagger} \| \, \| e \| \ .
  \]
If $m \geq n$ and $A$ and $B$ have full rank then
  \[
    \| x^* - \bar{x}^* \| \leq \| A_{\mathcal{N}(B)}^{\dagger} \| \, \| e \| \ .
  \]
When $B=A^T$ then the oblique pseudoinverse
$A_{\mathcal{N}(B)}^{\dagger}$ is the ordinary pseudoinverse $A^{\dagger}$
and we obtain the traditional least-squares perturbation bound.
\end{corollary}

\begin{proof}
The first bound is a direct consequence of \eqref{eq:fixpointnoisy}, and
the second bound follows from the fact that $\Poblb = I$ when
$m\geq n = \mathrm{rank}(A)$.
\end{proof}

The conclusions to be drawn from the analysis in this section is that the conditions
for the existence of a fixed point of the \BA\ are
rather strict, and that the fixed point is potentially very sensitive to data errors
since the matrix $A$ is ill-conditioned in inverse problems.
Moreover, it is very difficult to check the existence conditions
in a practical application, and it appears that they are very often
violated in the available software systems.
This motivates the development of a modified iterative method that is always
guaranteed to have a fixed point, which we introduce and analyze in the rest of
this paper.

\section{A Modified Algorithm}

In our numerical studies with ASTRA and other software packages for CT,
we have found that very often the convergence condition in \eqref{eq:BAcc}
is violated in that $B\, A$ has one or more eigenvalues with negative
real part.
As a consequence the iteration has no fixed point and the typical situation is
that the iterates $x^k$, after some iterations, start to diverge.
We illustrate this in Section~\ref{sec:numerical}.

\subsection{The Shifted BA Iteration}

To remedy this non-convergence issue, we propose a modified version
of the \BA\ that has guaranteed convergence and whose fixed point
approximates the exact solution $\bar{x}$.
In addition, the modified method should exhibit semi-convergence
properties similar to the \BA.
We refer to the modified algorithm as the \SBA,
and it guarantees convergence of the iterations
for appropriate choices of the two parameters
$\alpha > 0$ and $\omega> 0$:
  \begin{equation}
  \label{eq:ShiftedBAIteration}
    x^{k+1} = (1-\alpha\,\omega) \, x^k + \omega \, B \, (b - A\, x^k) \ , \qquad
    k = 0,1,2,\ldots \ .
\end{equation}
This scheme is motivated by the Tikhonov problem,
  \begin{equation}
  \label{eq:Tikhonov}
    x_{\alpha} = \arg \min_x \left\{ \| A \, x - b \|^2 +
    \alpha\, \| x \|^2 \right\} = ( A^TA + \alpha\, I )^{-1} A^T b \ ,
  \end{equation}
for which a gradient descent step takes the form
  \[
    x^{k+1} = x^k - \omega\, ( A^T (b - A \, x^k) + \alpha\, x^k ) =
    (1 - \alpha\,\omega)\, x^k + \omega\, A^T (b - A\, x^k) \ .
  \]
Hence, if $B=A^T$ then with a properly chosen $\omega$ the iteration
\eqref{eq:ShiftedBAIteration} converges to a Tikhonov solution $x_{\alpha}$.
Below we study the convergence properties of \eqref{eq:ShiftedBAIteration}
with $B \neq A^T$.

The matrix that governs the iterations for the \SBA\
(\ref{eq:ShiftedBAIteration}) is $B\, A + \alpha\, I$ with $I$ as the identity matrix,
whose eigenvalues are $\lambda_j + \alpha$
(where $\lambda_j$ are the eigenvalues of $BA$).
Our key idea is that by a proper choice of the additional
positive parameter $\alpha$
we can ensure that all these eigenvalues have a nonnegative real part\,---\,thus
ensuring convergence.
The shift needs to be just large enough that
$\Re(\lambda_j) + \alpha > 0$ for those $\lambda_j \neq -\alpha$.
At the same time, if $\alpha$ is small then the fixed point
will be an approximation to the exact solution~$\bar{x}$.
Hence, in contrast to the \BA\ the shifted version has a unique fixed point
that is always attained for both noisy and noise-free data.
Our new approach can therefore be viewed as a modification where both a
regularization term and semi-convergence of the iterations are used for noisy data.

We note that the \SBA\ is mathematically equivalent to applying the \BA\
to the augmented matrices and vector
 \begin{equation}\label{augMV}
   \begin{bmatrix} A \\ \sqrt{\alpha} I \end{bmatrix} , \qquad
   [ \, B \, , \, \sqrt{\alpha} I \, ] \ , \qquad
   \begin{bmatrix} b \\ 0 \end{bmatrix} .
 \end{equation}
A similar idea was used in \cite[\S 3.2]{ElHN12}, for the case $B=A^T$, to
perform convergence analysis for the case $\mathrm{rank}(A) < n$.
According to \cite[Props.~3.1 and 3.2]{ElHa18}, with the augmented matrices and vector defined in \eqref{augMV}
we obtain the following convergence criterion for the \SBA.

\begin{theorem}
\label{thm:rho}
Let $\lambda_j$ denote those eigenvalues of $BA$ that are different from $-\alpha$.
Then the \SBA\ \eqref{eq:ShiftedBAIteration}
converges to a fixed point if and only if $\alpha$ and $\omega$ satisfy
  \begin{equation}
  \label{eq:SBAcc}
	0 < \omega < 2 \ \frac{\Re(\lambda_j) + \alpha}{| \lambda_j|^2
      + \alpha \, (\alpha + 2 \, \Re(\lambda_j))} \qquad
    \hbox{and} \qquad \Re(\lambda_j) + \alpha > 0 \ .
  \end{equation}
\end{theorem}

\begin{proof}

Replacing the matrices $B$ and $A$ in the \BA\ with the
augmented ones from \eqref{augMV}, we define
$C = BA+\alpha I$ and $T=I-\omega C$. Then $T$ is the iteration matrix for the \SBA\ \eqref{eq:ShiftedBAIteration},
i.e.,
\begin{equation} \label {eq:Titer}
x^{k+1}=Tx^k+\omega Bb,
\end{equation}
and any fixed point $x^*_\alpha$ of
\eqref{eq:ShiftedBAIteration} satisfies the equation
\begin{equation}
\label{Cx=Bb}
C\, x^*_\alpha=B\, b \ .
\end{equation}

Let $P = P_{\mathcal{R}(C^T)}$ be the orthogonal projector onto $\mathcal{R}(C^T)
= \mathcal{N}(C)^\perp$. In view of the presence of the projection $P$, we consider the new coordinate system
given by the orthogonal matrix $\NR$,
where $N$ and $R$ are matrices with orthonormal columns spanning $\mathcal{N}(C)$
and $\mathcal{R}(C^T)$, respectively. We examine $P\, T$ in the new coordinates
where $\NR^T P\, \NR\ \NR^T \, T \, \NR$ takes the form
  \[
    \left[ \begin{array}{cc} 0 & 0 \\[2mm] 0 & I \end{array} \right] \
    \left[ \begin{array}{cc} I & -\omega \, N^TCR \\[2mm] 0 & I-\omega R^TCR
    \end{array} \right]
    = \left[ \begin{array}{cc} 0 & 0 \\[2mm] 0 & I-\omega R^TCR \end{array} \right].
  \]
For the first block row, it holds that
  \[
    Tx=x \quad \Leftrightarrow \quad C x = 0 \quad \Leftrightarrow
    \quad x\in \mathcal{N}(C) \quad \Leftrightarrow \quad 
    \hbox{$-\alpha$ is an eigenvalue of $BA$.}
  \]
This shows the eigenvalue $1$ of $T$ is associated with the eigenspace $\mathcal{N}(C)$.
We see that because of the projector operator, this eigenvalue,
which corresponds to eigenvalues of $BA$ equal to $-\alpha$, is irrelevant for convergence.
From the second block row, it suffices to consider $PT$ as operator
$\mathcal{R}(C^T) \to \mathcal{R}(C^T)$, where it holds that $PT = T$.
Combining these facts, we conclude that it is enough to consider the eigenvalues
$1-\omega (\lambda_j+\alpha)$ of $T$, where $\lambda_j$ is not equal to $-\alpha$.

Then, applying the results in \cite[Props.~3.1 and 3.2]{ElHa18} to the augmented matrices
and vector \eqref{augMV}, we obtain the sufficient and necessary condition of convergence
with respect to $\alpha$ and $\omega$.
\end{proof}

When we have convergence, a fixed point $x_{\alpha}^*$ of the \SBA\ satisfies
  \begin{equation}
  \label{eq:fispointshifted}
    x_{\alpha}^* = (1-\alpha\,\omega)\, x_{\alpha}^* + \omega\, B\, (b - A\, x^*)
    \qquad \Longleftrightarrow \qquad
    ( B\, A + \alpha\, I)\, x_{\alpha}^* = B\, b
  \end{equation}
and, similar to Theorem~\ref{prop:fixpoint}, we can characterize this fixed point as follows.

\begin{theorem}
\label{thm:fixpointshifted}
Assume that $B\, A + \alpha\, I$ is nonsingular.
The fixed point of the \SBA\ \eqref{eq:ShiftedBAIteration}
applied to $b$, with starting vector $x^0 = 0$ and $\alpha>0$, satisfies
  \begin{equation}
  \label{eq:fixpointshifted}
    x_{\alpha}^* = (B\, A + \alpha\, I)^{-1} B\, b =
    B\, (A\, B + \alpha\, I)^{-1} b \ , \qquad
    x_{\alpha}^* \in \mathcal{R}(B) \ .
  \end{equation}
For noise-free data $\bar{b}=A\,\bar{x}$ the fixed point
$\bar{x}_{\alpha}^*$ satisfies
  \[
    \bar{x}_{\alpha}^* = (B\, A + \alpha\, I)^{-1} B\, A\ \bar{x} \ , \qquad
    \bar{x}_{\alpha}^* \in \mathcal{R}(BA) \ , \qquad
    \bar{x} - \bar{x}_{\alpha}^* = \alpha\, (B\, A + \alpha\, I)^{-1} \bar{x} \ .
  \]
\end{theorem}

\begin{proof}
The first relation follows immediately from \eqref{eq:fispointshifted}.
To obtain the second relation we write $x_{\alpha}^* = B\, y$ and note that
$(B\, A + \alpha\, I)\, B\, y = B\, (A\, B + \alpha\, I)\, y$.
Hence $y$ must satisfy
  \[
    B\, \bigl( (A\, B + \alpha\, I) \, y - b \bigr) = 0 \ ,
  \]
and therefore $y$ has the general form with an arbitrary component
in the null space of $B\, (A\, B + \alpha\, I)$ as well as $B$:
  \[
    y = (A\, B + \alpha\, I)^{-1} b + z \ , \qquad
    z \in \mathcal{N}(B) \ ,
  \]
We note that $A\, B + \alpha\, I$ is nonsingular due to our
assumption that $B\, A + \alpha\, I$ is nonsingular, cf.\ \cite[Thm.~1.3.22]{HJ}.
Thus
  \[
    x_{\alpha}^* = B\, y = B\, \bigl( (A\, B + \alpha\, I)^{-1} b + z \bigr)
    = B\, (A\, B + \alpha\, I)^{-1} b \ ,
  \]
and it follows immediately that $x_{\alpha}^*\in\mathcal{R}(B)$.
The first results for $\bar{x}_{\alpha}^*$ follows immediately
from \eqref{eq:fixpointshifted}.  To show the second result let $B\, A$
have the eigendecomposition $B\, A = W\, \mathrm{diag}(\lambda_i)\, W^{-1}$; then
$(B\, A + \alpha\, I)^{-1} B\, A = W\,\mathrm{diag}(\lambda_i/(\lambda_i+\alpha))\, W^{-1}$
from which it follows that
$\bar{x}_{\alpha}^* \in \mathrm{span} \{ w_1,\ldots,w_{\mathrm{rank}(BA)} \}
= \mathcal{R}(BA)$.
The third result follows from the relation
  \[
    (B\, A + \alpha\, I)\, (\bar{x} - \bar{x}_{\alpha}^*) =
    (B\, A + \alpha\, I)\, (\bar{x} - (B\, A + \alpha\, I)^{-1} B\, A\,
    \bar{x} ) = \alpha\,\bar{x} \ .
  \]
\end{proof}

Note that the results in \eqref{eq:fixpointshifted} can also be derived by
applying the augmented matrices and vector defined in \eqref{augMV} to \eqref{eq:fixpointnoisy}.

To summarize, we formulated the convergence conditions for the \SBA\ in
terms of the shift $\alpha$ and the relaxation parameter $\omega$, and
we gave explicit expressions for the fixed point of this iterative method.

\subsection{First-Order Perturbation Analysis}
\label{sec:PerturbationAnalysis}

Recall that the fixed point $x_{\alpha}^*$ of the \SBA\
is the Tikhonov solution in \eqref{eq:Tikhonov} when $B=A^T$.
Following \cite{ElHa18} it is natural to give a general perturbation analysis of
the Tikhonov problem
when different perturbations are introduced in the matrices $A$ and $A^T$
in the corresponding normal equations $A^T A\, x = A^T b$.
A special instance of this analysis is when $B$ is an unmatched transpose of $A$,
and where our analysis lets us bound the difference between the fixed point
$x_{\alpha}^*$ and the exact solution $\bar{x}$.

We introduce the perturbed quantities
  \begin{equation}
  \label{eq:perturbations}
    \widetilde{A} = A + E_A, \qquad
    \UT = A^T + E_{A^T}, \qquad
	b = \bar{b} + e \ ,
  \end{equation}
with $E_A \in \R^{m \times n}$, $E_{A^T} \in \R^{n\times m}$ and
$e \in \R^m$.
Moreover we define $\tilde{x}_{\alpha}$ as the solution to
the \textsc{Regularized Unmatched Normal Equations}
  \[
    ( \UT \widetilde{A} + \alpha \, I) \, x = \UT b \ .
  \]
We want to compare $\tilde{x}_{\alpha}$ to the exact solution
$\bar{x}$.
To do this, we introduce the Tikhonov solution
$\bar{x}_{\alpha} = ( A^TA + \alpha\, I )^{-1} A^T \bar{b}$
to the unperturbed problem \eqref{eq:Tikhonov}.

We then split the error into a \textsl{perturbation error}
$\tilde{x}_\alpha - \bar{x}_\alpha$ and a
\textsl{regularization error} $\bar{x}_\alpha - \bar{x}$ as follows:
  \begin{equation}
  \label{eq:split}
    \tilde{x}_\alpha - \bar{x} = (\tilde{x}_\alpha - \bar{x}_\alpha) +
    (\bar{x}_\alpha - \bar{x}) \ .
  \end{equation}
This approach allows us to study the effect of the matrix and right-hand side
perturbations in isolation from the effect that Tikhonov regularization
has on a noise-free system.

\begin{theorem}
\label{prop:PerturbationError}
With the definitions in \eqref{eq:perturbations} and \eqref{eq:split} we have
the following first-order
error bounds obtained by omitting higher-order terms:
  \begin{align*}
	\| \tilde{x}_{\alpha} - \bar{x}_{\alpha} \| & \, \lesssim\,
      \frac{1}{2\sqrt{\alpha}} \, \| e \| + \frac{1}{2\sqrt{\alpha}} \,
      \| E_A \, \bar{x}_\alpha\| +
      \frac{1}{\alpha} \, \| E_{A^T} \, (\bar{b} - A\, \bar{x}_\alpha) \|
     \ , \\	
	\frac{\| \tilde{x}_\alpha - \bar{x}_\alpha \|}{\| \bar{x}_\alpha \|} & \, \lesssim\,
      \frac{\| A \|}{2\sqrt{\alpha}} \, \frac{\| e \|}{\| A\, \bar{x}_\alpha \|} +
      \frac{\| A \|}{2\sqrt{\alpha}} \, \frac{\| E_A \|}{\| A \|} \ +
      \frac{\| A \|^2}{\alpha} \frac{\| E_{A^T} \|}{\| A \|} \,
      \frac{\| \bar{b} - A\, \bar{x}_\alpha \|}{\| A\, \bar{x}_\alpha \|}
      \ .
  \end{align*}
\end{theorem}

\begin{proof}
Let $\tilde{x}_\alpha = \bar{x}_\alpha + \delta x_\alpha$
and consider the perturbed system
  \begin{equation}
  \label{eq:pertsyst}
    (\UT \widetilde{A} + \alpha\, I) \, (\bar{x}_\alpha + \delta x_\alpha) =
    \UT b \ .
  \end{equation}
Moreover note that from \eqref{eq:perturbations} we have
  \[
    \UT b = (A^T + E_{A^T})\, (\bar{b} + e) = A^T \bar{b} + A^T e + E_{A^T} b
  \]
and
  \[
    \UT \widetilde{A} = (A^T + E_{A^T}) \, (A + E_A) = A^T A + E \ ,
  \]
where we have introduced
  \[
    E = A^T E_A + E_{A^T} A+E_{A^T} E_A \ .
  \]
Inserting these equations in \eqref{eq:pertsyst} we obtain
  \[
    (A^T A + E + \alpha\, I) \, (\bar{x}_\alpha + \delta x_\alpha) =
    A^T \bar{b} + A^T e + E_{A^T} b \ ,
  \]
and rearranging we get
  \[
    (A^T A + \alpha\, I) \, \bar{x}_\alpha + (A^T A + \alpha\, I) \,
    \delta x_\alpha = A^T \bar{b} + A^T e + E_{A^T} \, b - E\, \bar{x}_\alpha -
    E \, \delta x_\alpha \ .
  \]
Now using that $(A^TA+\alpha I)\,\bar{x}_\alpha = A^T \bar{b}$ and neglecting
higher-order terms we get
  \begin{align*}
    (A^T A + \alpha I) \, \delta x_\alpha & \, \approx\,
      A^T e + E_{A^T} \, b - A^T E_A \, \bar{x}_\alpha - E_{A^T} \, A \, \bar{x}_\alpha \\
    & \, = \, A^T (e - E_A \, \bar{x}_\alpha) + E_{A^T} \, (\bar{b} - A\,\bar{x}_\alpha) \ ,
  \end{align*}
 which can also be obtained by using the augmented form in the proof of
 \cite[Prop.~2.1]{ElHa18}, i.e., replacing $A$, $E$, and $b$ with
 \[
   \begin{bmatrix} A \\ \sqrt{\alpha} I \end{bmatrix} , \qquad
   \begin{bmatrix} E \\ 0 \end{bmatrix} ,  \qquad
   \begin{bmatrix} b \\ 0 \end{bmatrix} .
 \] 
This leads to the bound
  \begin{align*}
    \| \delta x_\alpha \| & \, \lesssim\, \| (A^T A + \alpha\, I)^{-1} A^T \| \
      \| (e - E_A \, \bar{x}_\alpha) \| \ + \\
    & \qquad\quad \| (A^T A + \alpha\, I)^{-1} \|  \,
      \| E_{A^T} \, ( \bar{b} - A\, \bar{x}_\alpha ) \| \\
    &\leq \frac{1}{2\sqrt{\alpha}} \, \| e \| + \frac{1}{2\sqrt{\alpha}}\,
      \| E_A \, \bar{x}_\alpha\|
      + \frac{1}{\alpha}\,\| E_{A^T} \, ( \bar{b} - A\, \bar{x}_\alpha ) \| \ ,
  \end{align*}
where we use that, with $\sigma_i$ being the $i$th singular value of $A$,
  \[
    \| (A^TA + \alpha\, I)^{-1}A^T \| = \max_i
    \frac{\sigma_i}{\sigma_i^2+\alpha} \leq \frac{1}{2\sqrt{\alpha}}
  \]
and
  \[
    \| (A^T A + \alpha I)^{-1} \| = \max_i \frac{1}{\sigma_i^2+\alpha}
    \leq \frac{1}{\alpha} \ .
  \]
For the relative error we get
  \begin{align*}
    \frac{\| \delta x_\alpha \|}{\| \bar{x}_\alpha \|} & \, \lesssim\,
      \frac{1}{2\sqrt{\alpha}}\,\frac{\| e \|}{\| \bar{x}_\alpha \|} +
      \frac{1}{2\sqrt{\alpha}}\,
      \frac{\| E_A \, \bar{x}_\alpha \|}{\| \bar{x}_\alpha \|}
      + \frac{1}{\alpha}\,
      \frac{\|E_{A^T}\,(\bar{b}-A\,\bar{x}_\alpha)\|}{\| \bar{x}_\alpha \|}\\
    &\leq \frac{\| A \|}{2\sqrt{\alpha}}\,\frac{\| e \|}{\| A\,\bar{x}_\alpha \|} +
      \frac{\| A \|}{2\sqrt{\alpha}}\,\frac{\| E_A \|}{\| A \|} +
      \frac{\| A \|^2}{\alpha} \frac{\| E_{A^T} \|}{\| A \|}
      \, \frac{\| \bar{b} - A\,\bar{x}_\alpha \|}{\| A\,\bar{x}_\alpha \|} \ ,
  \end{align*}
where we used that
  \[
    \| A\,\bar{x}_\alpha \| \leq \| A \| \, \| \bar{x}_\alpha \| \quad
    \Longleftrightarrow \quad \frac{1}{\| \bar{x}_\alpha \|} \leq
    \frac{\| A \|}{\| A\,\bar{x}_\alpha \|} \ .
  \]
\end{proof}
To complete the analysis we need to bound the regularization error
$\bar{x}_\alpha - \bar{x}$ associated with the noise-free system.
To obtain a useful bound we need to incorporate the fact
that we solve a discretized inverse problem.
This is done in the following theorem from \cite{Hans90}
(see also \cite[Thm.~4.5.1]{HansenDIP}).

\begin{theorem}
Introduce SVD of $A$ as $A = \sum_{i=1}^{\min(m,n)} u_i\, \sigma_i\, v_i^T$
and assume that the noise-free right-hand side $\bar{b}$ is given by the model
  \[
	u_i^T \, \bar{b} = \sigma_i^\nu \ , \qquad i=1,2,\dots,\min(m,n) \ ,
  \]
in which $\nu\geq0$ is a model parameter that controls the decay of these coefficients.
Then
  \[
	\frac{\| \bar{x}_\alpha - \bar{x} \|}{\| \bar{x} \|} \leq
	\begin{cases}
	  \sqrt{n}, & 0 \leq \nu < 1 \ , \\[1mm]
	  \sqrt{n}\left( \displaystyle \frac{\sqrt{\alpha}}{\| A \|}
        \right)^{\nu-1}, & 1 \leq \nu < 3 \ , \\[4mm]
	  \sqrt{n}\left( \displaystyle \frac{\sqrt{\alpha}}{\| A \|}\right)^{\!2},
        & 3 \leq \nu \ .
	\end{cases}
  \]
\end{theorem}

\noindent
In practice the Tikhonov regularization parameter $\alpha$ is always
less than $\| A \|^2$~\cite{HansenDIP}.
This theorem then says that the noise-free problem must satisfy the discrete
Picard condition for Tikhonov regularization to produce a useful result\,---\,which is,
of course, the case for the imaging problems we have in mind.


To summarize our results, the shift parameter $\alpha$ plays the following roles.
The regularization error decreases as $\alpha$ decreases, and if the noise-free data
satisfies the discrete Picard condition (as we expect) then
a small nonzero $\alpha$ has little influence on the regularization error.
On the other hand, as $\alpha$ decreases then the perturbation error increases.
We want to use an $\alpha$ just large enough to ensure convergence.

\section{Eigenvalue Estimation}
\label{sec:eigenvalues}

The motivation behind the \SBA\ is to introduce a small
positive shift parameter $\alpha$,
just large enough to ensure that all the shifted eigenvalues have a positive real part,
i.e., $\Re(\lambda_j) + \alpha > 0$.
To turn this principle into an efficient working algorithm,
we therefore need to be able to estimate the leftmost eigenvalue $\lambda_{\mathrm{lm}}$
of $B\, A$, the eigenvalue with the minimal real part.
If $\Re(\lambda_{\mathrm{lm}}) > 0$ then we just use the \BA,
otherwise we use the \SBA\ with $\alpha$ slightly larger than $|\Re(\lambda_{\mathrm{lm}})|$.
(In view of Theorem~\ref{thm:rho}, we might theoretically even take $\alpha$ exactly equal
to this quantity, but this is not important in practice, since the approximation
to the smallest real part of the leftmost eigenvalue will usually be an upper bound.)
It is important to note that we only have actions with $A$ and $B$ at our disposal,
and no actions with $A^T$, $B^T$, or exact shift-and-invert transformations, are possible.

Various approaches have been investigated by Meerbergen and coauthors
for the rightmost eigenvalue of a matrix $C$
(or, equivalently, the leftmost of $-C$). Several of these are ``matrix-free'',
which means that only matrix-vector products are necessary.
An approach based on Chebyshev polynomials has been proposed in \cite{MRo94}.
In \cite{MSa99}, the search space is expanded by an approximation to
$\exp(C)\, z$, where $z$ is the current Ritz vector.
However, these methods usually take a considerable number of matrix-vector
multiplications to expand the search space by one vector.

Since the leftmost eigenvalue is an eigenvalue located at the exterior of the
spectrum, Krylov based methods are often well suited.
Stewart's Krylov--Schur method \cite{Ste02} is one of the most popular methods
to compute such eigenvalues.
This method is essentially equivalent to implicitly restarted Arnoldi \cite{Sor92},
as for instance implemented in Matlab's {\tt eigs}, but has a particularly
elegant and easy-to-understand implementation.
In our experiments, a custom-made implementation of the Krylov--Schur method
proved to be, on average, a factor 1.2--1.3 faster than {\tt eigs}
in terms of matrix-vector multiplications.
We give pseudocode for the Krylov--Schur method in Algorithm~1.

\noindent\vrule height 0pt depth 0.5pt width \textwidth \\
{\bf Algorithm~1: Krylov--Schur for the leftmost eigenvalue} \\[-3mm]
\vrule height 0pt depth 0.3pt width \textwidth \\
{\bf Input:} Minimal and maximal subspace dimensions $\underline{\ell} < \bar{\ell}$,
  starting vector $v_1$, tole-\\
\hspace*{12mm} rance \texttt{tol}, functions to perform matrix-vector
  multiplications with $A$ and $B$. \\
{\bf Output:} A pair $(\theta, v)$ that approximates the leftmost eigenpair
  $(\lambda_{\mathrm{lm}}, v_{\mathrm{lm}})$, \\
\hspace*{15mm} with $\|(B\,A-\theta I)\,v\| \le \texttt{tol}$. \\
\begin{tabular}{ll}
{\footnotesize 1:} & Form the Krylov decomposition $B\,A\,V_{\bar{\ell}} =
  V_{\bar{\ell}}\, H_{\bar{\ell}} + h_{\bar{\ell}+1,\bar{\ell}}\, v_{\bar{\ell}+1}\,
  f_{\bar{\ell}}^H$\\
{\footnotesize 2:} & {\bf for} $k = 1, 2, \dots$ \\
{\footnotesize 3:} & \phantom{M} Extract Schur pairs $(\theta_j, c_j)$ from
  $H_{\bar{\ell}}$ with $j=1, \dots, \bar{\ell}$, \\
& \phantom{MM} sorted on nondecreasing real part \\
{\footnotesize 4:} & \phantom{M} {\bf if}
  $| h_{\bar{\ell}+1,\bar{\ell}} \, f_{\bar{\ell}}^H\, c_1| \le \texttt{tol}$ \\
{\footnotesize 5:} & \phantom{MMM} Accept $\theta = \theta_1$ with leftmost
  Schur vector $v = V_{\bar{\ell}}\, c_1$, {\bf stop} \\
{\footnotesize 6:} & \phantom{M} {\bf end} \\
{\footnotesize 7:} & \phantom{M} Truncate decomposition to dimension $\underline{\ell}$
 by selecting leftmost Schur vectors \\
{\footnotesize 8:} & \phantom{M} Expand the Krylov decomposition to dimension $\bar{\ell}$ \\
{\footnotesize 9:} & {\bf end}\\
\end{tabular} \\
\vrule height 0pt depth 0.5pt width \textwidth

\medskip
This algorithm uses the Krylov decomposition from \cite{Ste02} which
is a generalization of the Arnoldi decomposition and which may have complex factors.
In Line~1, the first Krylov decomposition has $f_{\bar{\ell}} = e_{\bar{\ell}}$,
the ${\bar{\ell}}$th standard basis vector.
In subsequent Krylov decompositions this vector is changed, as indicated below.
In Line~4, we exploit the fact that
  \begin{align*}
    (B\,A-\theta I)\,v & = (B\,A-\theta I)\,V_{\bar{\ell}}\, c_1 \\
    & = V_{\bar{\ell}}\, (H_{\bar{\ell}}\, c_1-\theta_1\, c_1) +
    h_{{\bar{\ell}}+1,{\bar{\ell}}} \, v_{\bar{\ell}+1}\, f_{\bar{\ell}}^H\, c_1
    = h_{{\bar{\ell}}+1,{\bar{\ell}}} \,v_{\bar{\ell}+1}\, f_{\bar{\ell}}^H\, c_1.
  \end{align*}
As described in \cite{Ste02}, the restart in Lines 7--8 is performed as follows.
Suppose $H_{\bar{\ell}} = QSQ^H$ is the Schur decomposition with the most relevant
Schur vectors (corresponding to the leftmost Ritz values in $S$)
sorted in the beginning of $Q$.
Then the method is restarted with
$V_{\underline{\ell}} := V_{\bar{\ell}} \, Q_{1:\underline{\ell}}$
instead of $V_{\bar{\ell}}$;
$S_{\underline{\ell}}$ instead of $H_{\bar{\ell}}$; and
$Q_{1:\underline{\ell}}^Hf_{\bar{\ell}}$ instead of $f_{\bar{\ell}}$.
We present numerical experiments with this method in Section~\ref{sec:TestASTRA}.

An alternative approach that uses inexact shift-and-invert operators by carrying
out inner iterations is Jacobi--Davidson \cite{SVo00}.
This inexact inner-outer type of method may be worthwhile when the leftmost eigenvalue is not
well separated from neighboring eigenvalues.
In our applications this does not seem the case, and Jacobi--Davidson performs
usually worse than Krylov--Schur.
In addition, it is not obvious to generate a preconditioner for a shifted
version of $B\,A$, which would be very helpful for Jacobi--Davidson.

In our experiments with examples from computed tomography, we find that
the matrix $B\,A$ is often close to normal (in fact, even close to symmetric; 
cf.~Section~\ref{sec:TestASTRA}).
For such problems, another alternative approach to approximate the leftmost eigenvalue
is the following.
Let
  \[
    W(C) = \{ \, z^HCz \ : \ z \in \mathbb C^n, \ \|z\|=1 \, \}
  \]
be the field of values (or numerical range) of $C$.  Then we can expect the quantity
  \[
    \nu(BA) := \min \text{Re}(W(BA))
  \]
to be close to the leftmost eigenvalue of $B\, A$.
This $\nu(BA)$ would in principle be relatively easy to approximate, since it is equal to
$\frac{1}{2} \min \lambda(BA+(BA)^T)$, which results in computing an exterior eigenvalue
of a symmetric eigenproblem.
This would mean that we can use a symmetric version of Krylov--Schur,
which saves roughly half of the reorthogonalization costs.

Unfortunately, in our applications we do not have the action with $A^T$ or $B^T$,
and the described approach is not an option.
However, as an alternative, we can still approximate $\nu(BA)$ by
  \[
    \nu(BA) \approx \nu(H_{\bar{\ell}}) \ ,
  \]
where $H_{\bar{\ell}}$ is the matrix in the Krylov decomposition
obtained with $B\, A$ after several iterations.
The computation of this quantity only requires the known $H_{\bar{\ell}}^H$,
and therefore bypasses the need of the transposes of $A$ and $B$.
Although there are usually no error bounds for this type of approximation,
it may in practice be of very good quality.
This algorithm is summarized below.

\noindent\vrule height 0pt depth 0.5pt width \textwidth \\
{\bf Algorithm~2: A field of values approximation for the leftmost eigenvalue} \\[-3mm]
\vrule height 0pt depth 0.3pt width \textwidth \\
{\bf Input:} Minimal and maximal subspace dimensions $\underline{\ell} < \bar{\ell}$,
  starting vector $v_1$, \\
\hspace*{12mm} maximum iterations \texttt{maxit}, functions to perform actions with $A$ and $B$. \\
{\bf Output:} $\theta$, the leftmost point of a projected field of values, which approximates the \\
\hspace*{15mm} leftmost point of $W(BA)$. \\
\begin{tabular}{ll}
{\footnotesize 1:} & Form the Krylov decomposition $B\,A\,V_{\bar{\ell}} =
  V_{\bar{\ell}}\, H_{\bar{\ell}} + h_{\bar{\ell}+1,\bar{\ell}}\, v_{\bar{\ell}+1}\,
  f_{\bar{\ell}}^H$\\
{\footnotesize 2:} & {\bf for} $k = 1, 2, \dots, \texttt{maxit}$ \\
{\footnotesize 3:} & \phantom{M} Extract Schur pairs $(\theta_j, c_j)$ from
  $H_{\bar{\ell}}$
  with $j=1, \dots, \bar{\ell}$, \\
& \phantom{MM} sorted on nondecreasing real part \\
{\footnotesize 4:} & \phantom{M} Truncate decomposition to dimension $\underline{\ell}$
  by selecting leftmost Schur vectors \\
{\footnotesize 5:} & \phantom{M} Expand the Krylov decomposition to dimension $\bar{\ell}$ \\
{\footnotesize 6:} & {\bf end} \\
{\footnotesize 7:} & Accept $\theta = \min \Re W(H_{\bar{\ell}})$ =
  real part of leftmost eigenvalue of $\frac{1}{2} (H_{\bar{\ell}}+H_{\bar{\ell}}^H)$ \\
\end{tabular} \\
\vrule height 0pt depth 0.5pt width \textwidth

\medskip

Note that in a type of method as in Algorithm~2, there is typically no error estimate available;
there only is a user-chosen parameter \texttt{maxit}, which often can be modest (see Section~\ref{sec:numerical}).
A main advantage of Algorithm~2 over Algorithm~1 is that it may be possible to
stop the iterations improving the Krylov decomposition earlier, before the eigenpair
of Algorithm~1 has converged to the desired tolerance.
We test both approaches in the next section.

\section{Numerical Examples}
\label{sec:numerical}

We present numerical examples with two different test problems, in order to
demonstrate the performance of our new algorithm.
All computations are carried out in MATLAB.

\subsection{Small Illustrative Test Problem}

The first test problem is quite small, with $m=n=64$, such that we can
explicitly compute the desired eigenvalues and other quantities
that allow us to analyze the algorithms' performance related to
the above theory.
The matrix $A$ is full, and it is
generated by means of the function \texttt{regutm} from
\textsc{Regularization Tools} \cite{RT} by which we can generate random test matrices
with specified singular values, while the singular vectors have the
characteristic spectral behavior of inverse problems~\cite{Hans95}.
We generate a well-conditioned matrix $A_{\mathrm{well}}$ and a more
ill-conditioned matrix $A_{\mathrm{ill}}$ with the following distribution
of singular values:
\[
A_{\mathrm{well}}: \mathtt{logspace(0,-2,m)} \ , \qquad
A_{\mathrm{ill}}: \mathtt{logspace(0,-4,m)} \ .
\]
We then generate a corresponding unmatched transpose $\UT = A^T + E_{A^T}$ by adding
random Gaussian elements to $A^T$; all elements are from
$\mathcal{N}(0,\sigma_A^2)$ where the variance is chosen such that
$\| E_{A^T} \| \, / \, \| A \| = 0.05$.
Both $A$ and $\UT$ have full rank.

The exact solution $\bar{x}$ is the one from the \texttt{shaw} test problem
\cite{RT}; it is smooth with two humps.
Then we generate the exact and noisy right-hand sides by
\begin{equation}
\label{eq:bb}
\bar{b} = A\, \bar{x} \ , \qquad b = \bar{b} + e \ , \qquad
\| e \| \, / \, \| \bar{b} \| = 0.05 \ ,
\end{equation}
where the random elements of $e$ are Gaussian; all elements are from
$\mathcal{N}(0,\sigma_b^2)$ where the variance scales the noise as desired.

We applied both the \BA\ and the \SBA\
with $B = \UT$ to these problems.
For the \BA\ we use the relaxation parameter
$\omega = 1.9 \, / \, \| \UT A \|$ and for the \SBA\
we use $\omega$ equal to the upper bound in \eqref{eq:SBAcc} with the factor 2
replaced by 1.9.
For both systems we used \texttt{eig} to compute the eigenvalues;
the spectral radius is $\rho(BA) = 1.00$ by construction,
and the leftmost eigenvalues are
\begin{align*}
	\lambda_{\mathrm{lm}}(\UT_{\mathrm{well}} \, A_{\mathrm{well}}) &=
	4.78 \cdot 10^{-5} \ , \\
	\lambda_{\mathrm{lm}}(\UT_{\mathrm{ill}} \, A_{\mathrm{ill}}) &=
	- 2.02 \cdot 10^{-5} \pm i\ 3.83 \cdot 10^{-6} \ .
\end{align*}
Note that for $A_{\mathrm{ill}}$ the leftmost eigenvalue is a complex
conjugate pair with a negative real part.
Hence we expect the \BA\ to exhibit non-convergence for the
problems with $A_{\mathrm{ill}}$.
To ensure convergence of the \SBA\ we choose
\[
\alpha = 2\ | \Re ( \lambda_{\mathrm{lm}}(\UT A) ) | \ .
\]
\begin{figure}
	\begin{center}
		\includegraphics[width=0.48\textwidth]{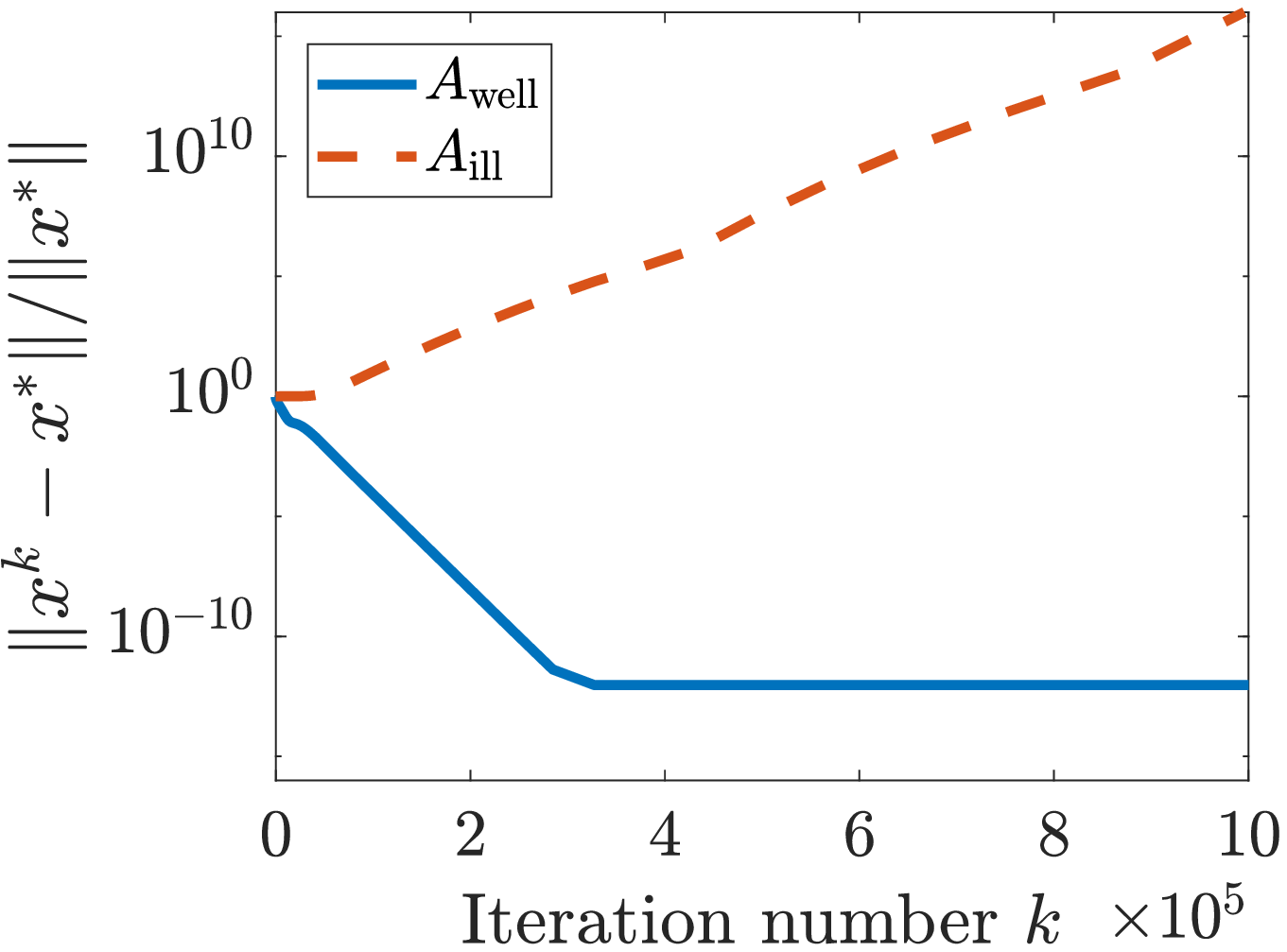}\quad
		\includegraphics[width=0.48\textwidth]{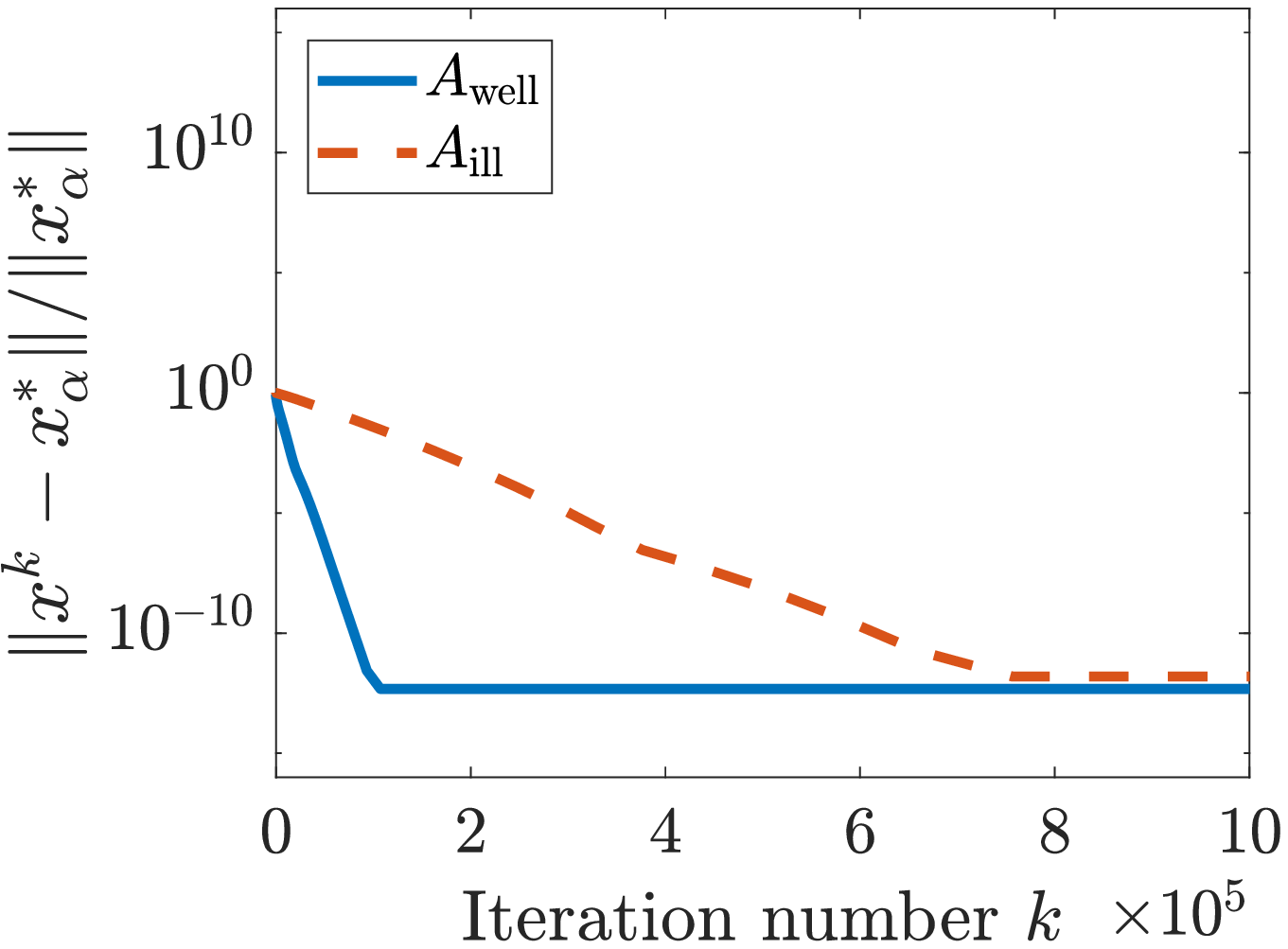}
	\end{center}
	\caption{\label{fig:smallconv} The convergence histories for the small
		test problems with noisy data ($e\neq 0$); the similar plots for noise-free
		data ($e=0$) are almost identical.
		Left:\ results for the \BA\ which exhibits non-convergence
		for the ill-conditioned matrix.
		Right:\ results for the \SBA\ with
		$\alpha =2\, | \Re ( \lambda_{\mathrm{lm}}(\UT A) ) |$; this method
		converges for both matrices.}
\end{figure}
The convergence histories for the noisy data ($e\neq 0$) are shown in
Figure~\ref{fig:smallconv}; the similar plots for the noise-free data ($e=0$)
are almost identical.
We make the following observations:
\begin{itemize}
	\item
	The left plot confirms that the \BA\ converges only for
	the well-conditioned matrix for which all eigenvalues have a positive real part,
	and that it converges to
	$x^* = A_{\mathcal{N}(B)}^{\dagger} b = B\, (AB)^{-1} b$.
	\item
	The right plot confirms that the \SBA\ converges for
	both matrices, and that it converges to
	$x_{\alpha}^* = B\, (AB+\alpha I)^{-1} b$.
    For the well-conditioned system the convergence is much faster
    compared to the ill-conditioned system.  Also,
    the \SBA\ converges faster than the \BA\ for the well-conditioned system.
\end{itemize}

\begin{figure}
	\begin{center}
		\includegraphics[width=0.48\textwidth]{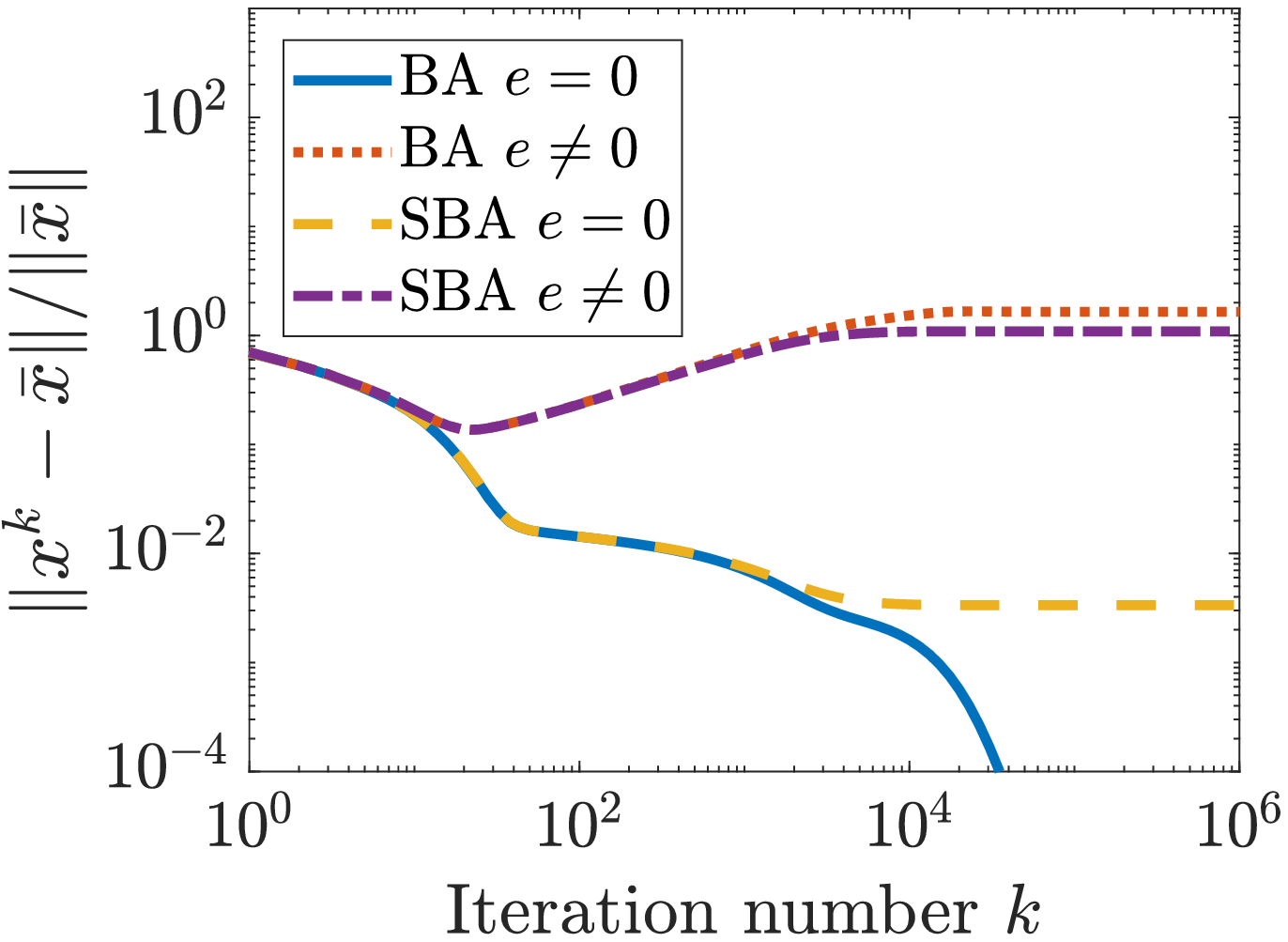}\quad
		\includegraphics[width=0.48\textwidth]{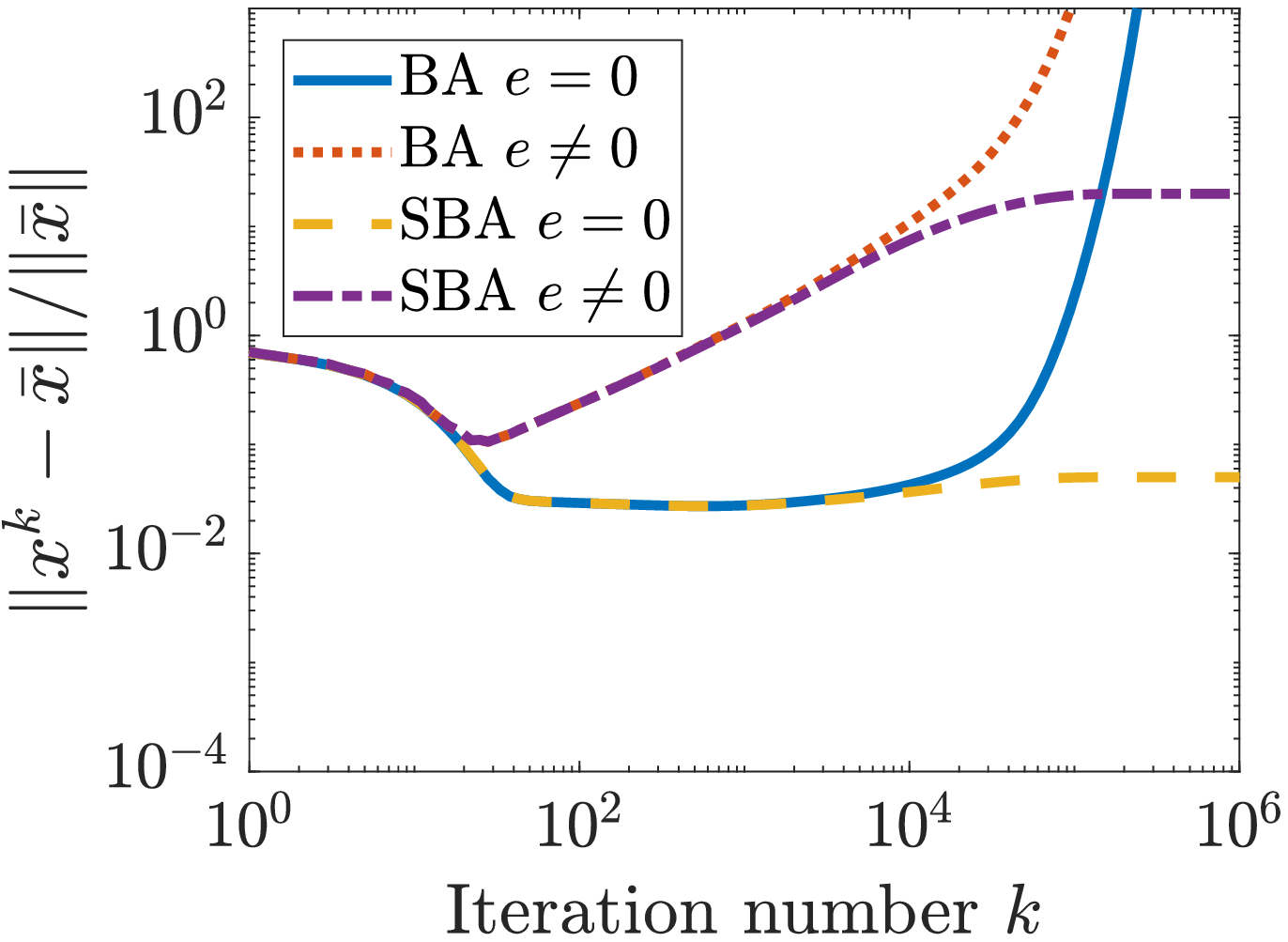}
	\end{center}
	\caption{\label{fig:smallsemiconv} The semi-convergence histories of the
		two iterative methods for $A_{\mathrm{well}}$ (left) and
		$A_{\mathrm{ill}}$ (right).
		For the problems with noisy data, we obtain a reasonably accurate
		solution at the point of semi-convergence.  See the text for a
		detailed discussion.}
\end{figure}

Having confirmed the convergence of the methods, it is also relevant to
study how well the methods are able to approximate the exact solution~$\bar{x}$.
To illustrate this, Figure~\ref{fig:smallsemiconv} shows plots of the
reconstruction errors $\| x^k - \bar{x}\| \, / \, \| \bar{x} \|$ versus iteration number $k$.
We make several observations:
\begin{itemize}
	\item
	For the well-conditioned matrix $A_{\mathrm{well}}$ and noise-free data ($e=0$)
	the \BA\ converges to the exact solution $\bar{x}$
	as predicted by Theorem~\ref{prop:fixpoint} with square and full-rank $A$ and $B$.
	\item
	For the ill-conditioned matrix $A_{\mathrm{ill}}$ the \BA\
	always diverges.
	\item
	For noise-free data ($e=0$)
	the \SBA\ converges to a slightly perturbed solution
	$\bar{x}_\alpha^\ast$ with
	\begin{align*}
		\| \bar{x}_\alpha^\ast - \bar{x} \| \, / \, \| \bar{x} \| &= 3.36 \cdot 10^{-3}
		\quad \hbox{for} \quad A_{\mathrm{well}} \ , \\[1mm]
		\| \bar{x}_\alpha^\ast - \bar{x} \| \, / \, \| \bar{x} \| & = 5.01 \cdot 10^{-2}
		\quad \hbox{for} \quad A_{\mathrm{ill}} \ .
	\end{align*}
	\item
	For noisy data ($e\neq 0$) the \BA\ for $A_{\mathrm{well}}$,
	as well as the \SBA\ for both $A_{\mathrm{well}}$
	and  $A_{\mathrm{ill}}$, converge to a fixed point that is quite far from
	the exact solution.
	Exactly the same behavior occurs for iterations that use a matching transpose.
	The main point is that for noisy data the iterations exhibit
	\textsl{semi-convergence}, where the iterates produce a good approximation to $\bar{x}$
	during the initial iterations.
\end{itemize}
In conclusion, these experiments verify the benefit of using the
\SBA, namely, guaranteed convergence while
retaining the semi-convergence that all algebraic iterative methods
rely on for noisy data.

\subsection{Test Problem From the ASTRA Toolbox}
\label{sec:TestASTRA}

The second test problem comes from X-ray computed tomography (CT) using a
parallel-beam geometry with 90 projections in the angular range
$0^{\circ}$--$180^{\circ}$, and a detector with 80 pixels and of
length is equal to the image size.
The exact solution $\bar{x}$ represents a $128 \times 128$ discretization
of the Shepp--Logan phantom.
Hence the problem size is $m=7200$ and $n=16384$, and the problem is underdetermined.
For such underdetermined CT problems the algebraic iterative methods
can give much better results than the ``standard'' methods based
on filtered backprojection \cite{NFD11}.
As before, the exact and noisy data are generated according to \eqref{eq:bb}.

The matrices $A$ and $\UT$ that represent the forward and backprojections
come from the ASTRA toolbox \cite{ASTRA} used in conjunction with
``spot operators'' \cite{SPOT}.
Specifically, we use the ASTRA function \texttt{opTomo} to compute the matrix-vector
multiplications with these matrices.
For the forward projection the GPU-version of ASTRA uses the interpolation model,
also known as Joseph's method \cite{Joseph}, while the backprojection
uses the line model with linear interpolation between detector pixels
(as done, e.g., in MATLAB's \texttt{iradon}).
The matrices are not stored;
if we store them then the sparsity of $A$ and $\UT$ is 1.32\% and 2.35\%,
respectively.
Measures of nonsymmetry and nonnormality of $BA$ are
  \[
    \frac{\| \frac{1}{2} \bigl( (BA) - (BA)^T \bigr) \|_{\mathrm{F}}}%
    {\| BA \|_{\mathrm{F}}} = 0.125 \ , \qquad
    \frac{\| (BA)\, (BA)^T - (BA)^T (BA) \|_{\mathrm{F}}}%
    {\| BA \|_{\mathrm{F}}^2} = 0.0235 \ .
  \]
The spectral radius is $\rho(BA) \approx 1.76 \cdot 10^4$.


For this test problem we use the algorithms from
Section \ref{sec:eigenvalues} to estimate the leftmost eigenvalue of $BA$,
and we compare the performance of the following strategies:
\begin{itemize}
\item \texttt{eigs} from MATLAB with options
  \texttt{maxit} = 1500, \texttt{tol} = $10^{-2}$, \texttt{SIGMA} = \texttt{'sr'};
\item \texttt{ks} is the Krylov--Schur method (Algorithm 1) with options
  \texttt{maxit} = 1500, absolute tolerance \texttt{tol} = $10^{-2}$,
  \texttt{mindim} = 30, \texttt{maxdim} = 60,
  \texttt{target} = \texttt{'-inf'} (for the leftmost eigenvalue);
\item \texttt{jd} is the Jacobi--Davidson method with options
  \texttt{maxit} = 1500, \texttt{tol} = $10^{-2}$, \texttt{mindim} = 30,
  \texttt{maxdim} = 60, \texttt{target} = \texttt{'-inf'};
\item \texttt{fovN} is the field of values based method (Algorithm 2) with options
  \texttt{maxit} = \texttt{N}, \texttt{mindim} = 30, \texttt{maxdim} = 60.
\end{itemize}
Table~\ref{tbl:eigenvalues} shows results for two cases:
\begin{enumerate}
\item
All matrix-vector multiplications are performed on the GPU, using the
ASTRA function {\tt opTomo}.
These multiplications are performed in \textsl{single precision}.
The Jacobi--Davidson method {\tt jd} did not converge, and this may be due to
the single-precision computations on the GPU.
\item
The two matrices $A$ and $\UT$ are explicitly computed and stored
as sparse matrices.  Like all the other computations, these computations
are performed on the CPU in \textsl{double precision}.
\end{enumerate}

\begin{table}
\caption{\label{tbl:eigenvalues} Estimation of the leftmost eigenvalue
$\lambda_{\mathrm{lm}}$ of $BA$ with the methods discussed in
Section~\ref{sec:eigenvalues}.
We use the ASTRA test problem mentioned in the text.
The methods stop when we reach convergence with absolute tolerance $10^{-2}$
or after the fixed number of iterations used in the field of values based method
Algorithm~2 {\tt fovN} for {\tt N} = 10, 15, and 20 iterations, respectively.
We show the mean number of matrix-vector multiplications (MVMs) as well as
the mean and standard deviation of the estimated $\lambda_{\mathrm{lm}}$.
The top half of the table shows results for the case when the matrix-vector
multiplications are done by the ASTRA functon {\tt opTomo}, while the
bottom half is for the case when we explicitly store the matrices.
The difference is due to the difference in precision (single vs.~double).
}
\begin{center}
\setlength{\tabcolsep}{5pt}
\renewcommand{\arraystretch}{1.1}
\begin{tabular}{l|cc} \hline
  & \multicolumn{2}{c}{25 trials with {\tt opTomo} that
          utilizes the GPU} \\
		& Mean MVM & Mean $\lambda_{\mathrm{lm}}$ (st.~dev.)  \\ \hline \rule{0pt}{2.3ex}%
		\texttt{fov10} & \phantom{1}660 & $-0.8921$ \ ($1.167 \cdot 10^{-1}$) \\
		\texttt{fov15} & \phantom{1}960 & $-0.9323$ \ ($7.478 \cdot 10^{-3}$) \\
		\texttt{fov20} &           1260 & $-0.9354$ \ ($3.579 \cdot 10^{-3}$) \\
		\texttt{ks}    &           1041 & $-0.9281$ \ ($2.510 \cdot 10^{-5}$) \\
		\texttt{eigs}  &           1278 & $-0.9281$ \ ($4.103 \cdot 10^{-5}$) \\ \hline
	\hline \rule{0pt}{2.8ex}%
		& \multicolumn{2}{c}{25 trials with $A$ and $\UT$ explicitly stored} \\
		& Mean MVM & Mean $\lambda_{\mathrm{lm}}$ (st.~dev.)  \\ \hline \rule{0pt}{2.3ex}%
		\texttt{fov10} & \phantom{1}660 & $-0.8920$ \ ($1.171 \cdot 10^{-1}$) \\
		\texttt{fov15} & \phantom{1}960 & $-0.9322$ \ ($7.647 \cdot 10^{-3}$) \\
		\texttt{fov20} &           1260 & $-0.9354$ \ ($3.581 \cdot 10^{-3}$) \\
		\texttt{ks}    &           1039 & $-0.9281$ \ ($2.334 \cdot 10^{-5}$) \\
		\texttt{eigs}  &           1261 & $-0.9281$ \ ($3.893 \cdot 10^{-5}$) \\
		\texttt{jd}    &           1404 & $-0.9281$ \ ($4.624 \cdot 10^{-5}$) \\ \hline
		\end{tabular}
\end{center}
\end{table}

\noindent
We see that for these CT problems the Krylov--Schur method
uses the least amount of computations, corresponding to the work in
performing about 500 iterations of the \SBA\ method.
When solving several CT problems with the same geometry, and hence the
same matrices, this is acceptable.
Even when an unmatched pair is used only once, this may be an acceptable
overhead to ensure convergence and trust in the computed solution.

\begin{figure}
	\begin{center}
		\includegraphics[width=0.55\textwidth]{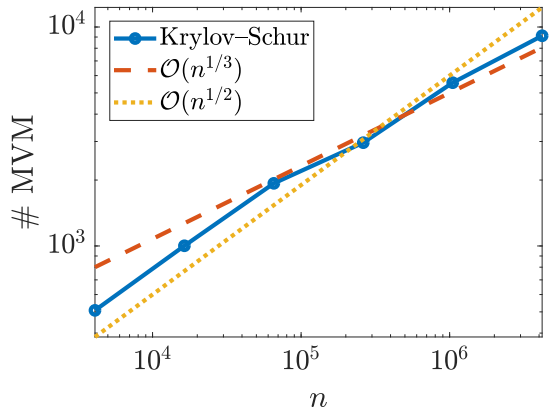}
	\end{center}
	\caption{\label{fig:ASTRAMVM} The number of matrix-vector multiplications
    (MVM) needed to estimate the leftmost eigenvalue with the Krylov-Schur method,
    as a function of the number of image pixels $n$.}
\end{figure}

Figure~\ref{fig:ASTRAMVM} reports the work involved in the eigenvalue estimation
with the Krylov-Schur method,
as measured by matrix-vector multiplications (MVMs), for different numbers
of image pixels $n = 64^2, 128^2, 256^2, 512^2, 1024^2, 2048^2$.
The number of rows is $m \approx 0.45\, n$.
We see that for larger problems the number of MVMs seems to be
proportional to $n^{1/3}$.

\begin{figure}
	\begin{center}
		\includegraphics[width=0.48\textwidth]{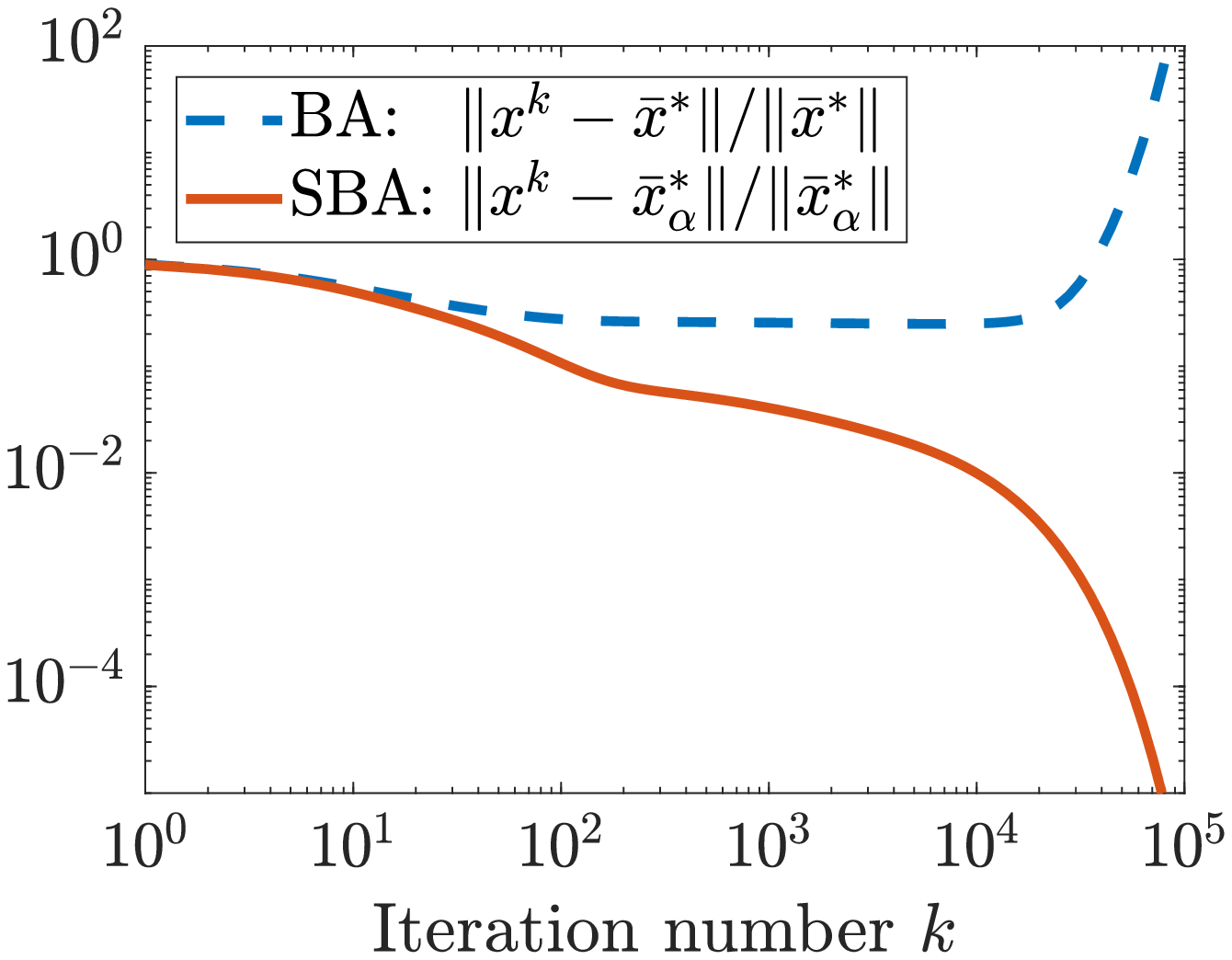}\quad
		\includegraphics[width=0.48\textwidth]{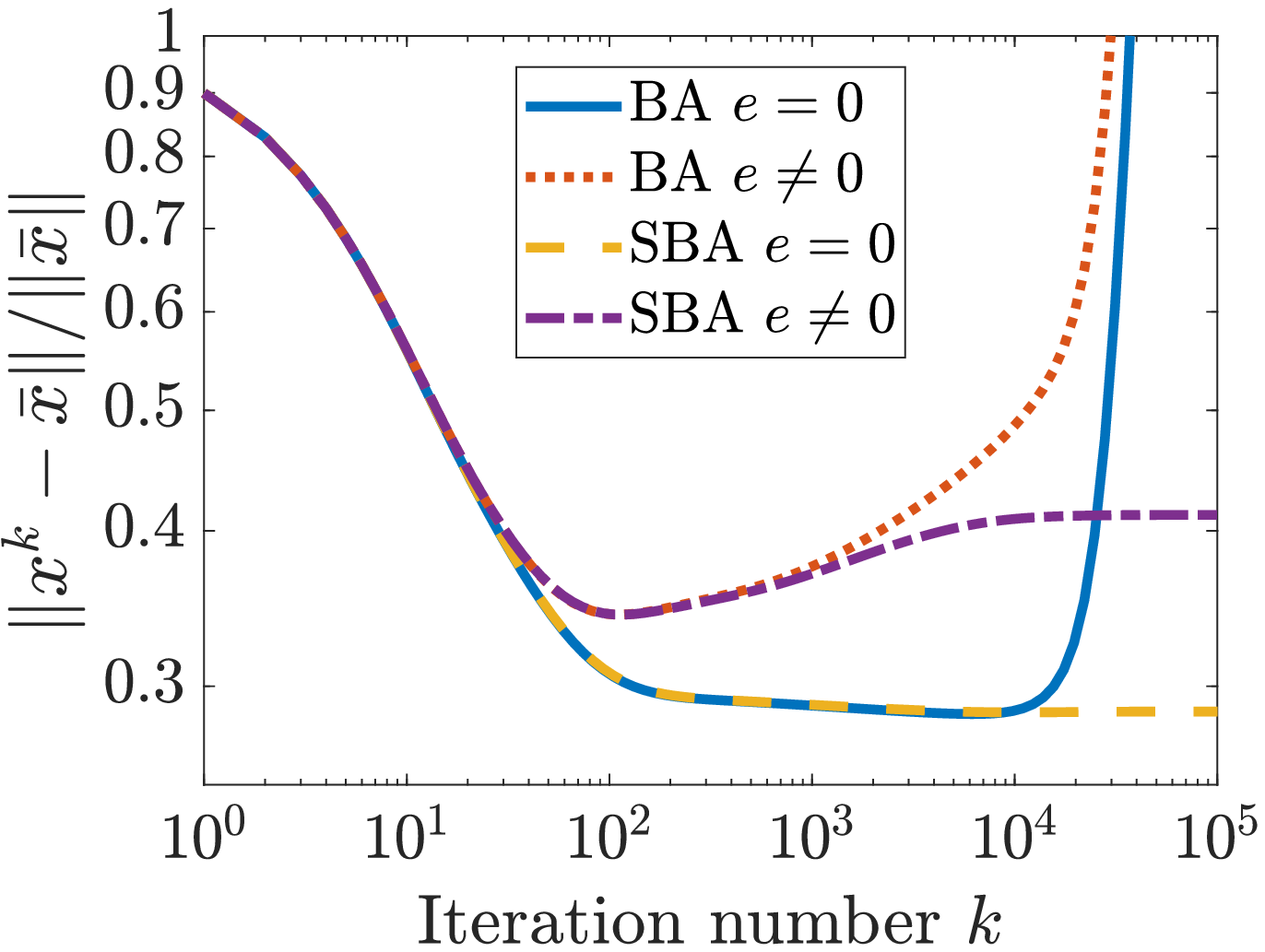}
	\end{center}
	\caption{\label{fig:ASTRA} Left:\ the convergence of the
		\BA\ and the \SBA\ for
		noise-free data ($e=0$); as expected the former does not converge.
	    Right:\ semi-convergence histories of the two iterative methods for
		noise-free as well as noisy data.}
\end{figure}

Figure~\ref{fig:ASTRA} shows the convergence histories when applying
the two iterative algorithms to this problem, with
$\alpha$ and $\omega$ chosen as before.
We observe the same behavior as before:
The \BA\ diverges, while the \SBA\
converges to a fixed point.
Moreover, the \SBA\ exhibits semi-convergence
as expected, and the minimum reconstruction error\,---\,at the point
of semi-convergence\,---\,does not deteriorate when using the shifted method.

\section{Conclusions}
\label{sec:conclusions}

We have considered algebraic iterative reconstruction methods with an
unmatched backprojector, i.e., the backprojector is not the exact adjoint
or transpose of the forward projector.
In particular we are concerned with the common situation where the iterative
method does not converge, due to the nonsymmetry of the iteration matrix.
We propose a modified algorithm that uses a small shift parameter,
we define the conditions that guarantee convergence to a
fixed point of a slightly perturbed problem,
and we give perturbation bounds for this fixed point.
We also discuss how to efficiently estimate the leftmost eigenvalue of a certain
matrix, which is needed to computed the shift parameter in the modified algorithm.
Numerical experiments with artificial test problems as well as a test
problem from computed tomography illustrate the use of the new algorithm.
Our MATLAB code is available on request.

\section*{Acknowledgements}
We thank Willem Jan Palenstijn (CWI) for sharing his insight into the
ASTRA package.
We acknowledge the inspiration from Tommy Elfving and Martin S.~Andersen who,
independently, suggested the shift as a remedy for the nonconvergence.
Finally, we thank two anonymous referees for many valuable comments that
helped to improve the paper.

\bibliographystyle{siamplain}

\end{document}